\documentclass[12pt]{article}
\usepackage{latexsym,amssymb,amsthm,enumerate,float,geometry,cite}
\newtheorem{theorem}{Theorem}
\newtheorem{proposition}[theorem]{Proposition}
\newtheorem{corollary}[theorem]{Corollary}

\newtheorem{claim}{Claim}

\usepackage[
        bookmarksopen=true,
        bookmarksopenlevel=1,
        colorlinks=true,
        linkcolor=blue,
        linktoc=page,
        citecolor=blue,
]{hyperref}

\usepackage{lineno}
\usepackage{setspace}

\begin{document}
\onehalfspace

\title{Low Weight Perfect Matchings}
\author{Stefan Ehard\thanks{Funded by the Deutsche Forschungsgemeinschaft (DFG, German Research Foundation) - 388217545}
\and Elena Mohr
\and Dieter Rautenbach}
\date{}

\maketitle
\vspace{-10mm}
\begin{center}
{\small 
Institute of Optimization and Operations Research, Ulm University,\\ 
Ulm, Germany, \texttt{$\{$stefan.ehard,elena.mohr,dieter.rautenbach$\}$@uni-ulm.de}}
\end{center}

\begin{abstract}
Answering a question posed by Caro, Hansberg, Lauri, and Zarb,
we show that for every positive integer $n$
and every function $\sigma\colon E(K_{4n})\to\{-1,1\}$ 
with $\sigma\left(E(K_{4n})\right)=0$, 
there is a perfect matching $M$ in $K_{4n}$
with $\sigma(M)=0$.
Strengthening a result of Caro and Yuster,
we show that for every positive integer $n$ 
and every function $\sigma\colon E(K_{4n})\to\{-1,1\}$ with 
$\left|\sigma\left(E(K_{4n})\right)\right|<n^2+11n+2,$
there is a perfect matching $M$ in $K_{4n}$
with $|\sigma(M)|\leq 2$.
Both these results are best possible.\\[3mm]
\end{abstract}

\section{Introduction}

In~\cite{cahalaza} Caro, Hansberg, Lauri, and Zarb 
considered connected graphs $G$ 
together with a function $\sigma\colon E(G)\to\{-1,1\}$ 
labeling the edges of $G$ with $-1$ or $+1$,
and they studied conditions that imply the existence 
of different types of spanning trees $T$ with 
$$|\sigma(E(T))|=\Big|\sum\limits_{e\in E(T)}\sigma(e)\Big|\leq 1.$$
As a variation of this problem, they ask whether,
for every positive integer $n$
and every labeling $\sigma\colon E(K_{4n})\to\{-1,1\}$ 
of the edges of the complete graph $K_{4n}$ of order $4n$
with $\sigma\left(E(K_{4n})\right)=0$, 
there is a perfect matching $M$ in $K_{4n}$
with $\sigma(M)=0$.
We answer their question in the affirmative.

\begin{theorem}\label{thm:main}
For every positive integer $n$
and every function $\sigma\colon E(K_{4n})\to\{-1,1\}$ 
with $\sigma\left(E(K_{4n})\right)=0$, 
there is a perfect matching $M$ in $K_{4n}$
with $\sigma(M)=0$.
\end{theorem}
Under the hypothesis of Theorem~\ref{thm:main},
the existence of a perfect matching $M$ in $K_{4n}$ 
with $|\sigma(M)|\leq 2$
already follows from more general results due to Caro and Yuster, cf.~Theorem~1.1 in~\cite{cayu}.
More precisely, Caro and Yuster showed that the weaker hypothesis
$\left|\sigma\left(E(K_{4n})\right)\right|\leq 2(4n-1)$ suffices
for the existence of such a perfect matching $M$ 
with $|\sigma(M)|\leq 2$.
As observed in~\cite{cahalaza},
for infinitely many positive integers $n$,
there are functions $\sigma\colon E(K_{4n})\to\{-1,1\}$ 
with $\sigma\left(E(K_{4n})\right)=4\sqrt{n}-2$
such that $\sigma(M)\not=0$ for every perfect matching $M$ in $K_{4n}$.
Slightly modifying their construction, we obtain the following proposition,
which implies that Theorem~\ref{thm:main} is best possible for infinitely many values of $n$.

\begin{proposition}\label{thm:optimality}
For infinitely many positive integers $n$, 
there is a function $\sigma\colon E(K_{4n})\to\{-1,1\}$ 
with $\sigma\left(E(K_{4n})\right)=2$
such that $\sigma(M)\not=0$
for every perfect matching $M$ in $K_{4n}$.
\end{proposition}
Considering the construction in the proof of Proposition~\ref{thm:optimality} suggests that 
zero weight perfect matchings 
are excluded rather by parity reasons 
than by the imbalance $\left|\sigma\left(E(K_{4n})\right)\right|$
of $\sigma$.
We confirm this with our second main result showing
that much weaker conditions on the imbalance imply 
the existence of low weight perfect matchings.

\begin{theorem}\label{thm:2}
For every positive integers $n$ and $k$
with $k\geq 2$,
and every function $\sigma\colon E(K_{4n})\to\{-1,1\}$ with 
$\left|\sigma\left(E(K_{4n})\right)\right|<
n(n-1)+k(6n-1)+k^2,$
there is a perfect matching $M$ in $K_{4n}$
with $|\sigma(M)|\leq 2k-2$.
\end{theorem}
For $k=2$, Theorem~\ref{thm:2} implies 
the following strengthening of the above-mentioned consequence 
of the result of Caro and Yuster.

\begin{corollary}\label{corollary1}
For every positive integer $n$ 
and every function $\sigma\colon E(K_{4n})\to\{-1,1\}$ with 
$\left|\sigma\left(E(K_{4n})\right)\right|<n^2+11n+2,$
there is a perfect matching $M$ in $K_{4n}$
with $|\sigma(M)|\leq 2$.
\end{corollary}
Both, Theorem~\ref{thm:2} and, hence, also 
Corollary~\ref{corollary1} are best possible.
If, for instance, 
$\sigma\colon E(K_{4n})\to\{-1,1\}$
is such that the graph
$\left(V(K_ {4n}),\sigma^{-1}(1)\right)$
consists of a clique of order $3n+2$ and $n-2$ isolated vertices,
then $|\sigma(M)|\geq 4$ for every perfect matching $M$ in $K_{4n}$
while $\left|\sigma\left(E(K_{4n})\right)\right|=n^2+11n+2$.

All proofs are given in the next section.

For a survey concerning related results,
we refer the reader to~\cite{ca}
and the introduction of~\cite{cahalaza}.

\section{Proofs}

We start with the proof of our first main result.

\begin{proof}[Proof of Theorem~\ref{thm:main}]
We suppose, for a contradiction, 
that $\sigma\colon E(K_{4n})\to\{-1,1\}$ 
is such that $\sigma\left(E(K_{4n})\right)=0$
but that $\sigma(M)\neq 0$ for every perfect matching $M$ in $K_{4n}$.
First, we consider the case $n=1$.
The edge set of $K_4$ is the union 
of three edge-disjoint perfect matchings 
$M_1$, $M_2$, and $M_3$.
Since $\sigma(M_i)\neq 0$ for every $i$,
we obtain $\sigma(M_i)\in \{ -2,2\}$,
which implies the contradiction
$\sigma(E(K_{4n}))=\sigma(M_1)+\sigma(M_2)+\sigma(M_3)\neq 0$.
Hence, we may assume that $n\geq 2$.
We call an edge $e$ 
a {\it plus-edge} if $\sigma(e)=1$, 
and a {\it minus-edge} if $\sigma(e)=-1$.
Since $\sigma(E(K_{4n}))=0$,
\begin{eqnarray}\label{e1}
\mbox{\it there are exactly 
$\frac{1}{2}{4n\choose 2}=4n^2-n$
plus-edges and minus-edges in $K_{4n}$, respectively.}
\end{eqnarray}
For a matching $M$, we denote by $M^+$ and $M^-$ 
the sets of plus-edges and minus-edges in $M$, respectively.  
We choose a perfect matching $M$ in $K_{4n}$ 
such that $|\sigma(M)|$ is as small as possible.
Possibly replacing $\sigma$ with $-\sigma$,
we may assume that $\sigma(M)>0$.
Since $M$ contains $2n$ edges, 
$\sigma(M)$ is even, which implies $\sigma(M)\geq 2$.

We start with some easy observations.

\begin{claim}\label{claim1}
For every two edges $e$ in $M^+$ and $f$ in $M^-$, 
there are no two disjoint minus-edges between $e$ and $f$.
In particular,
there are at most two minus-edges between $e$ and $f$.
\end{claim}
\begin{proof}[Proof of Claim~\ref{claim1}] 
If there are two disjoint minus-edges $e'$ and $f'$ 
between two edges $e\in M^+$ and $f\in M^-$, 
then 
the perfect matching 
$N=(M\setminus \{ e,f\})\cup \{ e',f'\}$
satisfies
$0\leq \sigma(N)=\sigma(M)-2<\sigma(M)$,
contradicting the choice of $M$.
\end{proof}

\begin{claim}\label{claim2}
There are two edges $e$ and $f$ in $M^+$ 
such that there exists a minus-edge between $e$ and $f$.
\end{claim}
\begin{proof}[Proof of Claim~\ref{claim2}] 
Suppose, for a contradiction, that there is no minus-edge in $V(M^+)$;
that is, more formally, the subgraph of $K_{4n}$ induced by 
the set $V(M^+)$ of vertices 
that are incident with a plus-edge from $M$
contains no minus-edge of $K_{4n}$.
For $m^+=|M^+|$, we have $m^+\geq n+1$.
By Claim~\ref{claim1},
at least half the $2m^+(4n-2m^+)$
edges between $V(M^+)$ and $V(M^-)$
are plus-edges, and, hence, 
the total number of plus-edges is at least
$${2m^+\choose 2}+\frac{1}{2}\cdot 2m^+(4n-2m^+)
=(4n-1)m^+\geq (4n-1)(n+1)>4n^2-n,$$
contradicting (\ref{e1}).
\end{proof}

\begin{claim}\label{claim3}
For every two edges $u_1u_2$ and $v_1v_2$ in $M^+$, 
if $u_1v_1$ is a minus-edge, 
then $u_2v_2$ is also a minus-edge.
Furthermore, $\sigma(M)=2$.
\end{claim}
\begin{proof}[Proof of Claim~\ref{claim3}] 
If $u_1u_2$ and $v_1v_2$ are two edges in $M^+$
such that $u_1v_1$ is a minus-edge
and $u_2v_2$ is a plus-edge,
then 
the perfect matching 
$N=(M\setminus \{ u_1u_2,v_1v_2\})\cup \{ u_1v_1,u_2v_2\}$
satisfies
$0\leq \sigma(N)=\sigma(M)-2<\sigma(M)$,
contradicting the choice of $M$.
This implies the first part of the statement.
Now, suppose, for a contradiction, 
that $\sigma(M)>2$.
Since $\sigma(M)$ is even, we have $\sigma(M)\geq 4$.
By Claim~\ref{claim2},
there are two edges $u_1u_2$ and $v_1v_2$ in $M^+$
such that $u_1v_1$ and $u_2v_2$ are both minus-edges.
Now, the perfect matching $N$ as above 
satisfies
$0\leq \sigma(N)=\sigma(M)-4<\sigma(M)$,
contradicting the choice of $M$.
This completes the proof of the claim.
\end{proof}
Since $\sigma(M)=2$,
the matching $M$ contains exactly $n+1$ plus-edges
and $n-1$ minus-edges.
We call a perfect matching $N$ in $K_{4n}$ {\it good}
if $\sigma(N)=\sigma(M)$.
If $N$ is a good matching,
then an edge $e^+$ in $N^+$ is called {\it special}
if there is no minus-edge in $V\left(N^+\setminus \{ e^+\}\right)$,
that is, all minus-edges in $V(M^+)$ are adjacent with $e^+$.

We distinguish the following two cases.

\medskip

\noindent {\bf Case 1.} {\it Every good matching contains a special edge.}

\medskip

\noindent Let $e^+$ be a special edge in $M$.

\begin{claim}\label{claim4}
For every edge $e$ in $M^+\setminus \{e^+\}$, 
there exist only minus-edges between $e^+$ and $e$.
\end{claim}
\begin{proof}[Proof of Claim~\ref{claim4}] 
Suppose, for a contradiction,
that there is a plus-edge between $e^+$ 
and some edge $e$ in $M^+\setminus \{e^+\}$.
By Claim~\ref{claim3},
there are at least two plus-edges between $e^+$ and $e$.
Since $e^+$ is special,
it follows that there are at most $4n-2$ minus-edges in $V(M^+)$.
Therefore, by Claim~\ref{claim1},
the total number of plus-edges is at least
$${2n+2\choose 2}-(4n-2)+\frac{1}{2}\cdot (2n+2)(2n-2)
=4n^2-n+1,$$
contradicting (\ref{e1}).
\end{proof}

\begin{claim}\label{claim5}
There is no plus-edge in $V(M^-)$.
Furthermore, there is an edge $e^-$ in $M^-$
such that 
there are 
exactly three plus edges between $e^+$ and $e^-$,
and, for every two edges $e$ in $M^+$
and $e'$ in $M^-$ with $(e,e')\neq (e^+,e^-)$,
there are exactly two plus-edges between $e$ and $e'$.
\end{claim}
\begin{proof}[Proof of Claim~\ref{claim5}] 
By Claim~\ref{claim4},
there are exactly $4n$ minus-edges in $V(M^+)$,
and, hence, exactly ${2n+2\choose 2}-4n$
plus-edges in $V(M^+)$.
By Claim~\ref{claim1},
there are at least $\frac{1}{2}\cdot (2n+2)(2n-2)$
plus-edges between $V(M^+)$ and $V(M^-)$.
Since 
$$\left({2n+2\choose 2}-4n\right)+\frac{1}{2}\cdot (2n+2)(2n-2)
=4n^2-n-1,$$
observation (\ref{e1}) implies the existence of exactly one further plus-edge
not yet accounted for.

Suppose, for a contradiction, that $V(M^-)$ contains a plus-edge $f_1$,
that is, there are two edges $e_1$ and $e_2$ in $M^-$
such that $f_1$ lies between $e_1$ and $e_2$.
Let $f_2$ be the edge between $e_1$ and $e_2$ that is disjoint from $f_1$.
Since, by (\ref{e1}), $f_1$ is the only plus-edge in $V(M^-)$,
it follows that $f_2$ is a minus-edge.
Let $e_3$ be in $M^+\setminus \{e^+\}$, 
and let $f_3$ and $f_4$ be disjoint minus-edges between $e_3$ and $e^+$.
Now, the perfect matching 
$N=(M\setminus \{ e^+,e_1,e_2,e_3\})\cup \{ f_1,f_2,f_3,f_4\}$
satisfies
$\sigma(N)=\sigma(M)-2=0$,
contradicting the choice of $M$.

It follows that there is no plus-edge in $V(M^-)$
and that there are exactly 
$\frac{1}{2}\cdot (2n+2)(2n-2)+1$
plus-edges between $V(M^+)$ and $V(M^-)$.
By Claim~\ref{claim1}, 
this implies that there is an edge $\hat{e}$ in $M^+$
and an edge $e^-$ in $M^-$ such that there are 
exactly three plus edges between $\hat{e}$ and $e^-$,
and, for every two edges $e$ in $M^+$
and $e'$ in $M^-$ with $(e,e')\neq (\hat{e},e^-)$,
there are exactly two plus-edges between $e$ and $e'$.
In order to complete the proof of the claim,
it remains to show that $\hat{e}=e^+$.
Suppose, for a contradiction, that $\hat{e}\not=e^+$.
Since $|M^+|=n+1\geq 3$, 
there is an edge $e_1$ in $M^+\setminus \{ e^+,\hat{e}\}$.
Let $f_1$ and $f_2$ be two disjoint minus-edges between $e^+$ and $e_1$,
and let $f_3$ and $f_4$ be two disjoint plus-edges between $\hat{e}$ and $e^-$.
The perfect matching 
$N=(M\setminus \{ e^+,\hat{e},e^-,e_1\})\cup \{ f_1,f_2,f_3,f_4\}$
satisfies
$\sigma(N)=\sigma(M)-2=0$,
contradicting the choice of $M$.
Hence, $\hat{e}=e^+$,
which completes the proof of the claim.
\end{proof}
Let $f^+$ and $f^-$ be two disjoint edges between $e^+$ and $e^-$
such that $f^+$ is a plus-edge and $f^-$ is a minus-edge.
The perfect matching 
$N=(M\setminus \{ e^+,e^-\})\cup \{ f^+,f^-\}$
satisfies $\sigma(N)=\sigma(M)$,
that is, $N$ is good.
By the assumption in Case 1, 
$N$ contains a special edge.
Since there are exactly three plus-edges 
between $f^+$ in $N^+$ and $f^-$ in $N^-$,
Claim~\ref{claim5} implies that $f^+$ is a special edge in $N$.
If $u^-$ is such that $\{ u^-\}=e^-\cap f^+$,
then Claim~\ref{claim4} implies that there are only minus-edges between $u^-$
and $V(M^+\setminus \{ e^+\})=V(N^+\setminus \{ f^+\})$.
Let $f_1$ be the plus-edge between $e^+$ and $e^-$ 
that is not incident with $u^-$.
Let $e\in M^+\setminus \{ e^+\}$.
Let $f_2$ and $f_3$ be two disjoint edges between $e^+\cup e^-$ 
and $e$ that are disjoint from $f_1$,
in particular, $f_2$ and $f_3$ are both minus-edges.
Now, the perfect matching 
$M'=(M\setminus \{ e^+,e^-,e\})\cup \{ f_1,f_2,f_3\}$
satisfies 
$\sigma(M')=\sigma(M)-2=0$,
contradicting the choice of $M$,
which completes the proof in this case.

\medskip

\noindent {\bf Case 2.} {\it Some good matching contains no special edge.}

\medskip

\noindent In this case, we may assume that the good matching $M$
chosen at the beginning of the proof contains no special edge.

First, suppose, for a contradiction,
that there are two edges $e_1$ in $M^+$ and $e_2$ in $M^-$
such that there are two disjoint plus-edges $f_1$ and $f_2$ 
between $e_1$ and $e_2$.
Since $e_1$ is not a special edge in $M$,
Claim~\ref{claim3} implies the existence of two edges 
$e_3$ and $e_4$ both in $M^+\setminus \{ e_1\}$
such that there are two disjoint minus-edges $f_3$ and $f_4$
between $e_3$ and $e_4$.
Now, the perfect matching 
$N=(M\setminus \{ e_1,e_2,e_3,e_4\})\cup \{ f_1,f_2,f_3,f_4\}$
satisfies $\sigma(N)=\sigma(M)-2=0$,
contradicting the choice of $M$.
Hence, there are no two such edges as $e_1$ and $e_2$.
Together with Claim~\ref{claim1}
this implies that there are exactly 
$\frac{1}{2}\cdot (2n+2)(2n-2)=2(n+1)(n-1)$
plus-edges as well as minus-edges 
between $V(M^+)$ and $V(M^-)$.

Next, suppose, for a contradiction, 
that the number of plus-edges in $V(M^-)$ is odd.
In this case, there are two edges $e_1$ and $e_2$ in $M^-$
and two disjoint edges $f_1$ and $f_2$ between $e_1$ and $e_2$
such that $f_1$ is a plus-edge and $f_2$ is a minus-edge.
By Claim~\ref{claim2} and Claim~\ref{claim3},
there are two edges $e_3$ and $e_4$ in $M^+$
and two disjoint minus-edges $f_3$ and $f_4$ between $e_3$ and $e_4$.
Again, the perfect matching 
$N=(M\setminus \{ e_1,e_2,e_3,e_4\})\cup \{ f_1,f_2,f_3,f_4\}$
satisfies $\sigma(N)=\sigma(M)-2=0$,
contradicting the choice of $M$.
Hence, the number of plus-edges in $V(M^-)$ is even, 
say equal to $2k$ for some integer $k$.
By (\ref{e1}), the number of minus-edges in $V(M^+)$ is
$$(4n^2-n)-2(n+1)(n-1)-\left({2n-2\choose 2}-2k\right)=4n+2k-1,$$
which is odd.
Nevertheless, by Claim~\ref{claim3},
the number of minus-edges in $V(M^+)$ is even,
which is a contradiction, and completes the proof.
\end{proof}
The following is based on a construction given at the end of~\cite{cahalaza}.

\begin{proof}[Proof of Proposition~\ref{thm:optimality}]
Let $n$ be such that there exists a positive even integer $k$ 
with $4n=k^2+4$. 
Let $(A,B)$ be a partition of the vertex set of $K_{4n}$ 
with $|A|=\frac{1}{2}(k^2+k)+2$ and $|B|=\frac{1}{2}(k^2-k)+2$. 
Now, we define a function $\sigma\colon E(K_{4n})\to\{-1,1\}$ 
such that all edges between $A$ and $B$ 
receive the value $1$ 
and all remaining edges receive the value $-1$.
Note that 
\begin{eqnarray*}
|\sigma^{-1}(1)| &=&
|A|\cdot |B| = \left(\frac{k^2+k}{2}+2\right)\left(\frac{k^2-k}{2}+2\right)
=
\frac{(k^2+4)(k^2+3)+4}{4}
=
\frac{1}{2}{4n\choose 2}+1 
\end{eqnarray*}
and thus,
$$\sigma\left(E(K_{4n})\right)
=|\sigma^{-1}(1)|-|\sigma^{-1}(-1)|
=\left(\frac{1}{2}{4n\choose 2}+1\right)-\left(\frac{1}{2}{4n\choose 2}-1\right)=2.$$
Now, suppose, for a contradiction, 
that $M$ is a perfect matching in $K_{4n}$ with $\sigma(M)=0$.
Clearly, $M$ contains $n$ plus-edges between $A$ and $B$. 
Hence, the number of vertices in $A$ 
that are not covered by a plus-edge in $M$ is
\begin{eqnarray*}
|A|-n=\left(\frac{k^2+k}{2}+2\right)-\frac{k^2+4}{4}
=\frac{k}{2}\left(\frac{k}{2}+1\right)+1.
\end{eqnarray*}
Since this is an odd number,
by construction, not all these vertices can be covered by minus-edges in $M$,
which is a contradiction, and completes the proof.
\end{proof}
For the proof of our second main result, Theorem~\ref{thm:2},
we need the following extremal result about matchings
due to Erd\H{o}s and Gallai \cite{erdo}.

\begin{theorem}[Erd\H{o}s and Gallai \cite{erdo}]\label{lemma1}
If $G$ is a graph of order $4n$ with matching number $n-k$ 
for some positive integer $k$, then the number of edges of $G$ is at most 
${4n\choose 2}-{3n+k\choose 2},$
with equality if and only if 
$G$ is the complement of the disjoint union 
of a complete graph of order $3n+k$
and $n-k$ isolated vertices.
\end{theorem}
We proceed to the proof of our second main result.

\begin{proof}[Proof of Theorem~\ref{thm:2}]
Let $\sigma\colon E(K_{4n})\to\{-1,1\}$ 
be such that $\sigma\left(E(K_{4n})\right)\geq 0$
and $|\sigma(M)|\geq 2k\geq 4$ 
for every perfect matching $M$ in $K_{4n}$.
As observed above, $|\sigma(M)|$ is even 
for every perfect matching $M$ in $K_{4n}$.
Therefore, it remains to show that 
$\sigma\left(E(K_{4n})\right)$
is at least the term stated in the theorem.

We distinguish the following two cases.

\medskip

\noindent {\bf Case 1.} {\it $\sigma(M)\leq 0$
for some perfect matching $M$ in $K_{4n}$.
}

\medskip

\noindent We choose a perfect matching $M$ in $K_{4n}$
with $\sigma(M)\leq 0$ 
such that $\sigma(M)$ is as large as possible.
Since $\sigma(M)\leq -2k\leq -4$,
we obtain $|M^-|\geq n+k$ and $|M^+|\leq n-k$,
where we use the notation and terminology 
as in the proof of Theorem~\ref{thm:main}.
The following two observations correspond to 
Claim~\ref{claim1} and Claim~\ref{claim2}
within the proof of Theorem~\ref{thm:main}.

If there is an edge $e$ in $M^+$ and an edge $f$ in $M^-$
such that there are two disjoint plus-edges $e'$ and $f'$ 
between $e$ and $f$,
then the perfect matching 
$N=(M\setminus \{ e,f\})\cup \{ e',f'\}$
satisfies $\sigma(M)<\sigma(N)=\sigma(M)+2<0$,
contradicting the choice of $M$.
Hence, between every edge in $M^+$ and every edge in $M^-$,
there are at least two minus-edges, 
which implies that at least half the edges between 
$V(M^+)$ and $V(M^-)$ are minus-edges.

If there is no plus-edge in $V(M^-)$,
then there are stricly more minus-edges than plus-edges,
contradicting $\sigma\left(E(K_{4n})\right)\geq 0$.
Hence, there are two edges $e$ and $f$ in $M^-$ 
such that there exists a plus-edge $e'$ between $e$ and $f$.
Let $f'$ be the edge between $e$ and $f$ that is disjoint from $e'$.
The perfect matching $N=(M\setminus \{ e,f\})\cup \{ e',f'\}$
satisfies  $\sigma(M)<\sigma(N)\leq \sigma(M)+4\leq -2k+4\leq 0$,
contradicting the choice of $M$.
This contradiction completes the proof in this case.

\medskip

\noindent {\bf Case 2.} {\it $\sigma(M)>0$
for every perfect matching $M$ in $K_{4n}$.}

\medskip

\noindent Note that $\sigma(M)\geq 2k$
for every perfect matching $M$ in $K_{4n}$.
Let $\nu$ be the matching number of the graph
$G=\left(V(K_{4n}),\sigma^{-1}(-1)\right)$.

Suppose, for a contradiction, that $\nu>n-k$.
If $M$ is a maximum matching in $G$,
then $V(G)\setminus V(M^-)$ is an independent set in $G$.
This implies that there is a matching $M^+$ in 
$\left(V(K_{4n}),\sigma^{-1}(1)\right)$
covering all vertices in $V(G)\setminus V(M^-)$.
Note that $|M^+|=2n-\nu<n+k$,
and, hence, $M^-\cup M^+$ is a perfect matching in $K_{4n}$
with $\sigma(M^-\cup M^+)<-(n-k)+(n+k)=2k$,
which is a contradiction.
Hence, we have $\nu\leq n-k$.

Let $\nu=n-k'$ for some integer $k'\geq k$.
By Theorem~\ref{lemma1}, we obtain
\begin{eqnarray*}
\sigma\left(E(K_{4n})\right) & = & {4n\choose 2}-2m(G)
\geq{4n\choose 2}-2\left({4n\choose 2}-{3n+k'\choose 2}\right)
=2{3n+k'\choose 2}-{4n\choose 2}\\
& \geq & 2{3n+k\choose 2}-{4n\choose 2}
=n(n-1)+k(6n-1)+k^2,
\end{eqnarray*}
which completes the proof.
\end{proof}

\end{document}